\documentclass[12pt]{amsart}

\usepackage{amsmath, amssymb,amsmath, enumerate,mathrsfs, graphicx, color}

\def\ZG{\mathbb{Z}G}
\def\Or{\mathcal{O}_{\mathcal F}(G)}
\DeclareMathOperator{\FA}{\mathsf{FV}}
\def\mc{\mathcal}
\def\cd{\mathsf{cd}}
\def\cdF{\mathsf{cd}_{\mc F}}
\def\gdF{\mathsf{gd}_{\mc F}}

\newtheorem{theorem}{Theorem}[section]
\newtheorem{lemma}[theorem]{Lemma}
\newtheorem{step}{Step}
\newtheorem{proposition}[theorem]{Proposition}
\newtheorem{corollary}[theorem]{Corollary}

\theoremstyle{definition}
\newtheorem{definition}[theorem]{Definition}

\newtheorem{question}[theorem]{Question}

\newtheorem{assumption}[theorem]{Assumption}

\begin{document}

\title[]{Dismantlable classifying space for the family of parabolic subgroups of a relatively hyperbolic group}
\author[Eduardo Martinez-Pedroza]{Eduardo Martinez-Pedroza}
   \address{Memorial University, St. John's, Newfoundland, Canada A1C 5S7 }
  \email{emartinezped@mun.ca}
\author[Piotr Przytycki]{Piotr Przytycki}
   \address{McGill University, Montreal, Quebec, Canada H3A 0B9}
 \email{piotr.przytycki@mcgill.ca}

\begin{abstract}
Let $G$ be a group hyperbolic relative to a finite collection of subgroups $\mc P$.
Let $\mathcal F$ be the family of subgroups consisting of all the conjugates of subgroups in $\mc P$,
all their subgroups, and all finite subgroups. Then there is a cocompact model
for $E_{\mc F} G$. This result  was known in the torsion-free case.  In the presence of torsion, a new approach was necessary. Our method is to exploit the notion of dismantlability.  A number of sample  applications are discussed.
\end{abstract}

\maketitle

\section{Introduction}
Let $G$ be a finitely generated group hyperbolic relative to a
finite collection $\mc P=\{P_\lambda\}_{\lambda\in \Lambda}$ of
its subgroups (for a definition see Section~\ref{sec:Rips}).
Let $\mathcal F$ be the collection of all the conjugates of
$P_\lambda$ for $\lambda\in\Lambda$, all their subgroups, and
all finite subgroups of $G$. \emph{A model for $E_{\mc F}G$} is
a $G$-complex~$X$ such that all point stabilisers belong
to $\mc F$, and for every $H\in \mc F$ the fixed-point set
$X^H$ is a (nonempty) contractible subcomplex of $X$. A model for $E_{\mc F}G$
is also called the \emph{classifying space for the family $\mc F$}. In this article we describe a particular classifying space for the family $\mc F$. It admits the following simple description.

Let $S$ be a finite generating set of $G$. Let $V=G$ and let
$W$ denote the set of cosets $gP_\lambda$ for $g\in G$ and
$\lambda\in \Lambda$. We consider the elements of $W$ as
subsets of the vertex set of the Cayley graph of $G$ with
respect to $S$. Then $|\cdot , \cdot|_S$, which denotes the
distance in the Cayley graph, is defined on $V\cup W$. The
\emph{$n$-Rips graph} $\Gamma_n$ is the graph with vertex set
$V\cup W$ and edges between $u,u'\in V\cup W$ whenever
$|u,u'|_S\leq n$. The \emph{$n$-Rips complex}
$\Gamma_n^\blacktriangle$ is obtained from $\Gamma_n$ by
spanning simplices on all cliques. It is easy to prove that
$\Gamma_n$ is a fine $\delta$-hyperbolic connected graph (see
Section~\ref{sec:Rips}). Our main result is the following.

\begin{theorem}
\label{thm:MMain} For $n$ sufficiently large, the $n$-Rips
complex $\Gamma_n^\blacktriangle$ is a cocompact model for
$E_{\mc F}G$.
\end{theorem}

Theorem~\ref{thm:MMain} was known to hold if
\begin{itemize}
\item $G$ is a torsion-free hyperbolic group and $\mc
    P=\emptyset$, since in that case the $n$-Rips complex
    $\Gamma_n^\blacktriangle$ is contractible for $n$
    sufficiently large~\cite[Theorem 4.11]{ABC91}.
\item $G$ is a hyperbolic group and $\mc P=\emptyset$,
    hence $\mathcal F$ is the family of all finite
    subgroups, since in that case $\Gamma_n^\blacktriangle$
    is $\underline{E}(G)$~\cite[Theorem 1]{MS02}, see
    also~\cite[Theorem 1.5]{HOP14} and~\cite[Theorem 1.4]{La13}.
\item $G$ is a torsion-free relatively hyperbolic group,
    but with different definitions of the $n$-Rips complex,
    the result follows from the work of Dahmani~\cite[Theorem 6.2]{Da03-2},
    or Mineyev and Yaman~\cite[Theorem 19]{MiYa07}.
\end{itemize}
In the presence of torsion, a new approach was necessary. Our
method is to exploit the notion of \emph{dismantlability}.
Dismantlability, a property of a graph guaranteeing strong
fixed-point properties (see~\cite{Po93})  was brought to geometric
group theory by Chepoi and Osajda~\cite{ChOs15}. Dismantlability was
observed for hyperbolic groups in \cite{HOP14}, following
the usual proof of the contractibility of the Rips complex \cite[Prop III.$\Gamma$ 3.23]{BrHa99}.

While we discuss the $n$-Rips complex only for finitely
generated relatively hyperbolic groups, Theorem~\ref{thm:MMain}
has the following extension.

\begin{corollary}
\label{cor:infinite} If $G$ is an infinitely generated group
hyperbolic relative to a finite collection $\mc P$, then there
is a cocompact model for $E_{\mc F}G$.
\end{corollary}
\begin{proof}
By~\cite[Theorem 2.44]{Os06}, there is a finitely generated
subgroup $G'\leq G$ such that $G$ is isomorphic to $G'$
amalgamated with all $P_\lambda$ along $P'_\lambda=P_\lambda
\cap G'$. Moreover, $G'$ is hyperbolic relative to
$\{P_\lambda'\}_{\lambda\in \Lambda}$. Let $S$ be a finite
generating set of $G'$. While $S$ does not generate $G$, we can
still use it in the construction of
$X=\Gamma^\blacktriangle_n$. More explicitly, if $X'$ is the
$n$-Rips complex for $S$ and $G'$, then $X$ is a tree of copies
of $X'$ amalgamated along vertices in $W$. Let $\mathcal{F}'$
be the collection of all the conjugates of $P'_\lambda$, all
their subgroups, and all finite subgroups of $G'$. By
Theorem~\ref{thm:MMain}, we have that $X'$ is a cocompact model
for $E_{\mc F'}G'$, and it is easy to deduce that $X$ is a
cocompact model for $E_{\mc F}G$.
\end{proof}

\subsection*{Applications}

On our way towards Theorem~\ref{thm:MMain} we will establish
the following, for the proof see Section~\ref{sec:Rips}. We learned from
Fran\c{c}ois Dahmani that this corollary can be also obtained from one of
Bowditch's approaches to relative hyperbolicity.

\begin{corollary}\label{sec:subgroups} There is finite collection of finite
subgroups $\{F_1,\ldots, F_k\}$ such that any finite subgroup
of $G$ is conjugate to a subgroup of some $P_\lambda$ or some
$F_i$.
\end{corollary}

Note that by \cite[Theorem 2.44]{Os06},
Corollary~\ref{sec:subgroups} holds also if $G$ is infinitely
generated, which we allow in the remaining part of the
introduction. 

Our next application regards the cohomological dimension of
relatively hyperbolic groups in the framework of Bredon
modules. Given a group $G$ and a nonempty family $\mc F$
of subgroups closed under conjugation and taking subgroups,
the theory of (right) modules over the orbit category
$\mc{O_F}(G)$ was established by Bredon~\cite{Br67}, tom
Dieck~\cite{tD87}  and L{\"u}ck~\cite{Lu89}. In the case where
$\mc F$ is the trivial family, the $\Or$-modules are
$\ZG$-modules. The notions of cohomological dimension $\cdF(G)$
and finiteness properties $FP_{n, \mc F}$ for the pair $(G, \mc
F)$ are defined analogously to their counterparts $\cd(G)$ and
$FP_n$. The geometric dimension $\gdF(G)$ is defined as the
smallest dimension of a model for $E_{\mc F}G$. A theorem of
L\"uck and Meintrup~\cite[Theorem 0.1]{LuMe00} shows that  \[
\cdF(G) \leq \gdF(G) \leq \max\{3, \cdF(G)\}.\] Together with
Theorem~\ref{thm:MMain}, this yields the following. Here as before $\mathcal F$ is the collection of all the conjugates of $\{P_\lambda\}$, all their subgroups, and
all finite subgroups of $G$.

\begin{corollary} Let $G$ be relatively hyperbolic. Then $\cdF(G)$ is finite.
\end{corollary}

The \emph{homological Dehn
function} $\FA_X(k)$ of a simply-connected cell complex $X$
measures the difficulty of filling cellular $1$-cycles with
$2$-chains. For a
finitely presented group $G$ and $X$ a model for $EG$ with
$G$-cocompact $2$-skeleton, the growth rate of
$\FA_{G}(k):=\FA_X(k)$ is a group invariant~\cite[Theorem
2.1]{Fl98}.
The function $\FA_G(k)$ can also be defined  from algebraic considerations under the weaker assumption that $G$ is $FP_2$, see~\cite[Section 3]{HaMa15}.
Analogously, for a group $G$ and a family of subgroups $\mc F$ with a cocompact model for $E_{\mc F}G$, there is relative homological Dehn function $\FA_{G,\mc F}(k)$  whose growth rate
is an invariant of the pair $(G, \mc F)$, see~\cite[Theorem 4.5]{MP15-2}.

Gersten proved that a group $G$ is hyperbolic if and only if it is $FP_2$ and the growth rate of $\FA_G(k)$ is linear~\cite[Theorem 5.2]{Ge96}. An  analogous characterisation for relatively hyperbolic groups is proved in~\cite[Theorem 1.11]{MP15-2} relying on the following corollary. We remark that a converse of  Corollary~\ref{thm:char} requires an additional condition that $\{P_\lambda\}$ is an almost malnormal collection, see~\cite[Theorem 1.11(1) and Remark 1.13]{MP15-2}.

\begin{corollary}\label{thm:char}
Let $G$ be relatively hyperbolic. Then $G$ is $FP_{2, \mc F}$ and $\FA_{G, \mc F}(k)$ has linear growth.
\end{corollary}
\begin{proof}
The existence of a cocompact model
$X=\Gamma^\blacktriangle_n$ for $E_{\mc F}(G)$ implies that $G$
is $FP_{2,\mc F}$. Since $X$ has fine and hyperbolic
$1$-skeleton and has finite edge $G$-stabilisers, it follows that $\FA_{G,\mc F}(k):=\FA_X(k)$ has
linear growth by~\cite[Theorem 1.7]{MP15}.
\end{proof}

\noindent \textbf{Organisation}. In Section~\ref{sec:Rips} we
discuss the basic properties of the $n$-Rips complex
$\Gamma_n^\blacktriangle$, and state our main results on the
fixed-point sets, Theorems~\ref{thm:fixpoint}
and~\ref{thm:contractible}. We prove them in
Section~\ref{sec:proofs} using a graph-theoretic notion called
dismantlability. We also rely on a thin triangle
Theorem~\ref{thm:stronglythin} for relatively hyperbolic
groups, which we prove in Section~\ref{sec:triangle}.

\medskip

\noindent \textbf{Acknowledgements.} We thank Damian Osajda for
discussions, and the referees for comments. 
Both authors acknowledge funding by NSERC.
The second author was also partially supported by
National Science Centre DEC-2012/06/A/ST1/00259 and UMO-2015/\-18/\-M/\-ST1/\-00050.

\section{Rips complex}
\label{sec:Rips}

\subsection{Rips graph}

We introduce relatively hyperbolic groups following Bowditch's
approach~\cite{Bo12}. A \emph{circuit} in a graph is an
embedded closed edge-path. A graph is \emph{fine} if for every
edge $e$ and every integer $n$, there are finitely many
circuits of length at most $n$ containing $e$.

Let $G$ be a group, and let $\mc
P=\{P_\lambda\}_{\lambda\in \Lambda}$ be a finite collection of
subgroups of $G$. A \emph{$(G, \mc P)$-graph} is a fine
$\delta$-hyperbolic connected graph with a $G$-action with
finite edge stabilisers, finitely many orbits of edges, and
such that $\mc P$ is a set of representatives of distinct
conjugacy classes of vertex stabilisers such that each infinite
stabiliser is represented.

Suppose $G$ is finitely generated, and let $S$ be a finite generating set.
Let $\Gamma$
denote the Cayley graph of $G$ with respect to $S$. Let $V$
denote the set of vertices of~$\Gamma$, which is in
correspondence with $G$. A \emph{peripheral left coset} is a
subset of $G$ of the form $gP_\lambda$. Let $W$ denote the set
of peripheral left cosets, also called \emph{cone vertices}.
The \emph{coned-off Cayley graph} $\hat \Gamma$ is the
connected graph obtained from $\Gamma$ by enlarging the vertex
set by $W$ and the edge set by the pairs $(v,w)\in V\times W$,
where the group element $v$ lies in the peripheral left coset
$w$.

We say that $G$ is \emph{hyperbolic relative to $\mc P$}
if $\hat \Gamma$ is fine and $\delta$-hyperbolic, which means
that it is a \emph{$(G, \mc P)$-graph}. This is equivalent to the existence of a $(G, \mc P)$-graph. Indeed, if there is a $(G, \mc P)$-graph, a construction of Dahmani~\cite[Page 82, proof of Lemma 4]{Da03} (relying on an argument of Bowditch~\cite[Lemma 4.5]{Bo12}) shows that $\hat \Gamma$ is a $G$-equivariant subgraph of a $(G, \mc P)$-graph $\Delta$, and therefore $\hat \Gamma$ is fine and quasi-isometric to $\Delta$.
In particular the definition of relative hyperbolicity is independent of $S$.  From here on, we assume that $G$ is hyperbolic relative to $\mc P$.

Extend the word metric (which we also call \emph{$S$-distance})
$|\cdot , \cdot|_S$ from $V$ to $V\cup W$ as follows: the
distance between cone vertices is the distance in
$\Gamma$ between the corresponding peripheral left cosets, and
the distance between a cone vertex and an element of $G$ is the
distance between the corresponding peripheral left coset and
the element. Note that $|\cdot,\cdot|_S$ is not a metric on
$V\cup W$. It is only for $v\in V$ that we have the triangle
inequality $|a,b|_S\leq|a,v|_S+|v,b|_S$ for any $a,b\in V\cup
W$.

\begin{definition}
Let $n\geq 1$. The \emph{$n$-Rips graph} $\Gamma_n$ is the
graph with vertex set $V\cup W$ and edges between $u,u'\in
V\cup W$ whenever $|u,u'|_S\leq n$.
 \end{definition}

\begin{lemma}\label{lem:GPgraph}
The $n$-Rips graph $\Gamma_n$ is a $(G, \mc P)$-graph.
\end{lemma}
\begin{proof}
Note that the graphs $\hat \Gamma$ and $\Gamma_n$ have the same
vertex set. In particular since $\hat \Gamma$ is connected and
contained in $\Gamma_n$, it follows that $\Gamma_n$ is
connected.

Since $\Gamma$ is locally finite and there are finitely many
$G$-orbits of edges in $\Gamma$, it follows that  there are
finitely many $G$-orbits of edge-paths  of length $n$ in
$\Gamma$. Since $\mc P$ is finite, there are finitely many $G$-orbits of
edges in~$\Gamma_n$.

Since $G$ acts on $\hat \Gamma$ with finite edge stabilisers
and $\hat \Gamma$ is fine, it follows that for distinct
vertices in $V\cup W$, the intersection of their
$G$-stabilisers is finite~\cite[Lemma 2.2]{MW11}. Thus the
pointwise $G$-stabilisers of edges in $\Gamma_n$ are finite,
and hence the same holds for the setwise $G$-stabilisers of
edges.

It remains to show that $\Gamma_n$ is fine and
$\delta$-hyperbolic. Since there are finitely many $G$-orbits
of edges in $\Gamma_n$, the graph $\Gamma_n$ is obtained from
$\hat \Gamma$ by attaching a finite number of $G$-orbits of
edges. This process preserves fineness by a result of
Bowditch~\cite[Lemma 2.3]{Bo12}, see also~\cite[Lemma
2.9]{MW11}. This process also preserves the quasi-isometry
type~\cite[Lemma 2.7]{MW11}, thus $\Gamma_n$ is
$\delta$-hyperbolic.
\end{proof}

For a graph $\Sigma$, let $\Sigma^\blacktriangle$ be the
simplicial complex obtained from $\Sigma$ by spanning simplices
on all cliques. We call $\Gamma_n^\blacktriangle$ the
\emph{$n$-Rips complex}.

\begin{corollary}\label{cor:barycenters}
The $G$-stabiliser of a barycentre of a simplex $\Delta$ in
$\Gamma_n^\blacktriangle$ that is not a vertex is finite.
\end{corollary}
\begin{proof}
Let $F$ be the stabiliser of the barycentre of $\Delta$. Then
$F$ contains the pointwise stabiliser of $\Delta$ as a finite
index subgroup. The latter is contained in the pointwise
stabiliser of an edge of $\Delta$, which is finite by
Lemma~\ref{lem:GPgraph}. Therefore $F$ is finite.
\end{proof}

\begin{corollary}\label{cor:cocompact}
The $G$-action on $\Gamma_n^\blacktriangle$ is cocompact.
\end{corollary}
\begin{proof}
By Lemma~\ref{lem:GPgraph}, the $n$-Rips graph $\Gamma_n$ is
fine. Hence every edge~$e$ in $\Gamma_n$ is contained in
finitely many circuits of length $3$. Thus $e$ is contained in
finitely many simplices of $\Gamma^\blacktriangle_n$. By
Lemma~\ref{lem:GPgraph}, there are finitely many $G$-orbits of
edges in $\Gamma_n$.  It follows that there are finitely many
$G$-orbits of simplices in $\Gamma^\blacktriangle_n$.
\end{proof}

\subsection{Fixed-point sets}

The first step of the proof of Theorem~\ref{thm:MMain} is the
following fixed-point theorem.

\begin{theorem}\label{thm:fixpoint} For sufficiently large $n$, each finite subgroup
$F\leq G$ fixes a clique of $\Gamma_n$.
\end{theorem}

The proof will be given in Section~\ref{sub:quasi}. As a
consequence we obtain the following.

\begin{proof}[Proof of Corollary~\ref{sec:subgroups}]
By Corollary~\ref{cor:cocompact}, there are finitely many
$G$-orbits of simplices in~$\Gamma_n^\blacktriangle$. From each
orbit of simplices that are not vertices pick a
simplex~$\Delta_i$, and let $F_i$ be the stabiliser of its
barycentre. By Corollary~\ref{cor:barycenters}, the group $F_i$
is finite.

Choose $n$ satisfying Theorem~\ref{thm:fixpoint}. Then any
finite subgroup $F$ of $G$ fixes the barycentre of a simplex
$\Delta$ in $\Gamma^\blacktriangle_n$. If $\Delta$ is a vertex,
then $F$ is contained in a conjugate of some $P_\lambda$.
Otherwise, $F$ is contained in a conjugate of some $F_i$.
\end{proof}

It was observed by the referee that if one proved in advance Corollary~\ref{sec:subgroups}, one could deduce from it Theorem~\ref{thm:fixpoint} (without control on~$n$).

The second step of the proof of Theorem~\ref{thm:MMain} is the
following, whose proof we also postpone, to
Section~\ref{sub:contr}.

\begin{theorem}
\label{thm:contractible} For sufficiently large $n$, for any subgroup
$H \leq G$, its fixed-point set in $\Gamma_n^\blacktriangle$ is
either empty or contractible.
\end{theorem}

We conclude with the following.

\begin{proof}[Proof of Theorem~\ref{thm:MMain}]
The point stabilisers of $\Gamma_n^\blacktriangle$ belong to $\mc
F$ by Corollary~\ref{cor:barycenters}. For every $H\in \mc F$
its fixed-point set $(\Gamma_n^\blacktriangle)^H$ is nonempty
by Theorem~\ref{thm:fixpoint}. Consequently,
$(\Gamma_n^\blacktriangle)^H$ is contractible by
Theorem~\ref{thm:contractible}.
 \end{proof}

\section{Dismantlability}
\label{sec:proofs}

The goal of this section is to prove
Theorems~\ref{thm:fixpoint} and~\ref{thm:contractible}, relying
on the following.

\subsection{Thin triangle theorem}

We state an essential technical result of the article, a thin
triangles result for relatively-hyperbolic groups. We keep the notation from the Introduction, where $\Gamma$ is the Cayley graph of $G$ with respect to $S$ etc. By \emph{geodesics} we mean geodesic edge-paths.

\begin{definition}\cite[Definition 8.9]{HK08}
Let $p=(p_j)_{j=0}^\ell$ be a geodesic in $\Gamma$, and let
$\epsilon, R$ be positive integers. A vertex $p_i$ of $p$ is
\emph{$(\epsilon, R)$-deep} in the peripheral left coset $w\in
W$ if $R\leq i \leq \ell- R$ and  $|p_j,  w|_S \leq \epsilon$
for all $|j-i|\leq R$. If there is no such $w\in W$, then $p_i$
is an \emph{$(\epsilon, R)$-transition vertex} of~$p$.
\end{definition}

\begin{lemma}
\label{rem:atmostone} \label{lem:uniqueness}
 \cite[Lemma~8.10]{HK08}
For each $\epsilon$ there is a constant $R$ such that for any
geodesic $p$ in $\Gamma$, a vertex of $p$ cannot be $(\epsilon,
R)$-deep in two distinct peripheral left cosets.
\end{lemma}

\begin{definition} For $a,b \in V\cup W$, a
\emph{geodesic from $a$ to $b$ in $\Gamma$} is a geodesic in
$\Gamma$ of length $|a,b|_S$ such that its initial vertex
equals $a$ if $a\in V$, or is an element of $a$ if $a\in W$,
and its terminal vertex equals $b$ if $b\in V$, or is an
element of $b$ if $b\in W$.
\end{definition}

Throughout the article we adopt the following convention. For
an edge-path $(p_j)_{j=0}^\ell$, if $i>\ell$ then $p_i$ denotes
$p_\ell$.

\begin{theorem}[Thin triangle theorem]\label{thm:stronglythin}
There are positive integers $\epsilon, R$ and~$D$, satisfying
Lemma~\ref{rem:atmostone}, such that the following holds. Let
$a,b,c\in V\cup W$ with $a\neq b$, and let $p^{ab}, p^{bc},
p^{ac}$ be geodesics in $\Gamma$ from $a$ to $b$, from $b$ to
$c$, and from $a$ to $c$, respectively. Let $\ell=|a,b|_S$ and
let $0\leq i \leq \ell$.

If $p^{ab}_i$ is an $(\epsilon, R)$-deep vertex of $p^{ab}$ in
the peripheral left coset $w$, then let $z=w$, otherwise let
$z=p^{ab}_i$.

Then $|z, p^{ac}_i|_S\leq D$ or $|z, p^{bc}_{\ell-i}|_S \leq
D$.
\end{theorem}

Note that the condition $a\neq b$ is necessary, since for $a=b\in W$ we could take for $p^{ab}$ any element of $a$, leading to counterexamples.

While Theorem~\ref{thm:stronglythin} seems similar to various
other triangle theorems in relatively hyperbolic groups, its proof
is surprisingly involved, given that we rely on these previous results.
We postpone the proof till Section~\ref{sec:triangle}. In the remaining part of the section, $\epsilon, R,D$ are the integers guaranteed by Theorem~\ref{thm:stronglythin}. We can
and will assume that $D\geq \epsilon$.

\subsection{Quasi-centres}
\label{sub:quasi}

In this subsection we show how to deduce
Theorem~\ref{thm:fixpoint} from thin triangle
Theorem~\ref{thm:stronglythin}. This is done analogously as for
hyperbolic groups, using quasi-centres (see
\cite[Lemma~III.$\Gamma$.3.3]{BrHa99}).

\begin{definition}
Let $U$ be a finite subset of $V\cup W$. The \emph{radius}
$\rho(U)$ of~$U$ is the smallest $\rho$ such that there exists
$z\in V\cup W$ with $|z,u|_S\leq \rho$ for all $u\in U$. The
\emph{quasi-centre} of $U$ consists of $z\in V\cup W$ satisfying
$|z,u|_S\leq \rho(U)$ for all $u\in U$.
\end{definition}

\begin{lemma}
\label{lem:quasicentre} Let $U$ be a finite subset of $V\cup W$
that is not a single vertex of $W$. Then for any two elements
$a,b$ of the quasi-centre of~$U$, we have $|a,b|_S\leq 4D$.
\end{lemma}
\begin{proof} Assume first $\rho(U)\leq D$. If $U$ is a single vertex $v\in
V$, then $|a,b|_S\leq |a,v|_S+|v,b|_S\leq 2D$ and we are done.
If there are $u\neq u'\in U$, then let $v$ be the first vertex
on a geodesic from $u$ to $u'$ in~$\Gamma$. Since the
vertex $v$ is not $(\epsilon, R)$-deep, by
Theorem~\ref{thm:stronglythin} applied to $u,u', a$
(respectively $u,u',b$), we obtain a vertex $v_a$ (respectively
$v_b$) on a geodesic in $\Gamma$ from $a$ (respectively $b$)
to $u$ or $u'$ satisfying $|v_a,v|_S\leq D$ (respectively
$|v_b,v|_S\leq D$). Consequently $|a,b|_S\leq
|a,v_a|_S+|v_a,v|_S+|v,v_b|_S+|v_b,b|_S\leq 2\rho(U)+2D\leq
4D$, as desired.

Henceforth, we assume $\rho(U)\geq D+1$. Let $\ell=|a,b|_S$. If
$\ell> 4D$, then choose any $2D\leq i\leq \ell-2D$, and any
geodesic $p^{ab}$ from~$a$ to~$b$ in~$\Gamma$. If $p^{ab}_i$ is
an $(\epsilon, R)$-deep vertex of $p^{ab}$ in the peripheral
left coset~$w$, then let $z=w$, otherwise let $z=p^{ab}_i$. We
claim that for any $c\in U$ we have $|z,c|_S\leq \rho(U)-1$,
contradicting the definition of $\rho(U)$, and implying
$\ell\leq 4D$.

Indeed, we apply Theorem~\ref{thm:stronglythin} to $a,b,c,$ and
any geodesics $p^{bc},p^{ac}$. Without loss of generality
assume that we have $|z,p^{ac}_{i}|_S\leq D$. Note that if
$i>|a,c|_S$, then $p^{ac}_i$ lies in (or is equal to) $c$, and
consequently $|z,c|_S\leq |z,p^{ac}_{i}|_S\leq D\leq\rho(U)-1$,
as desired. If $i\leq|a,c|_S$, then $|p^{ac}_i,a|_S=i\geq 2D$,
and hence
$$|z,c|_S\leq
|z,p^{ac}_{i}|_S+|p^{ac}_{i},c|_S\leq
D+(|c,a|_S - |p^{ac}_i,a|_S)\leq D+ (\rho(U)-2D),$$ as desired.
\end{proof}

\begin{proof}[Proof of Theorem~\ref{thm:fixpoint}]
Let $n\geq 4D$. Consider a finite orbit $U\subset V\cup W$ of
$F$. If $U$ is a single vertex, then there is nothing to prove.
Otherwise, by Lemma~\ref{lem:quasicentre} the quasi-centre of
$U$ forms a fixed clique in~$\Gamma_n$.
\end{proof}

\subsection{Convexity}

\begin{definition}
Let $\mu$ be a positive integer. A subset $U\subset V\cup W$ is
\emph{$\mu$-convex} with respect to $u\in U$ if for any
geodesic $(p_j)_{j=0}^\ell$ in $\Gamma$ from $u$ to $u'\in U$,
for any $j\leq \ell-\mu$, we have
\begin{enumerate}[(i)]
\item $p_j\in U$, and
\item for each $w\in W$ with $|w,p_j|_S\leq \epsilon$ we
    have $w\in U$.
\end{enumerate}
\end{definition}

\begin{definition}
\label{def:conv}
Let $r$ be a positive integer and let $U\subset V\cup W$ be a
finite subset. The \emph{$r$-hull} $U_r$ of $U$ is the union of
\begin{enumerate}[(i)]
\item all the vertices $v\in V$ with $|v,u|_S\leq r$ for
    each $u\in U$, and
\item all the cone vertices $w\in W$ with $|w,u|_S\leq
    r+\epsilon$ for each $u\in U$.
\end{enumerate}
\end{definition}

\begin{lemma}
\label{rem:finiteU} If $|U|\geq 2$, then each $U_r$ is finite.
\end{lemma}
\begin{proof}
Choose $u\neq u'\in U$. Assume without loss of generality
$r\geq |u,u'|_S$. Each vertex of $U_r$ distinct from $u$ and
$u'$ forms with $u$ and $u'$ a circuit in $\Gamma_{r+\epsilon}$
of length $3$. There are only finitely many such circuits,
since $\Gamma_{r+\epsilon}$ is fine by Lemma~\ref{lem:GPgraph}.
\end{proof}

\begin{lemma}
\label{lem:hull} Let $U\subset V\cup W$ be a finite subset with
$|u,u'|_S\leq \mu$ for all $u,u'\in U$. Then each $U_r$, with
$r\geq \mu$, is $(\mu+2D)$-convex with respect to all $b\in U$.
\end{lemma}

\begin{proof}
Let $b\in U$, and $a\in U_r$. Let $(p^{ab}_j)_{j=0}^\ell$ be a geodesic from
$a$ to $b$ in~$\Gamma$, and let $\mu+2D\leq j\leq \ell$. By Definition~\ref{def:conv}, we
have $\ell\leq r+\epsilon$. To prove the lemma, it suffices to
show that $p^{ab}_j\in U_r$.

Consider any $c\in U$ and apply Theorem~\ref{thm:stronglythin}
with $i=\ell$. In that case $p_i^{ab}$ is not $(\epsilon,
R)$-deep and thus $z=p_i^{ab}$. Consequently, we have
$|p_\ell^{ab} , p^{ac}_\ell|_S\leq D$ or $|p_\ell^{ab} ,
p^{bc}_{0} |_S\leq D.$

In the second case, using $\epsilon\leq D$, we have
$$|p^{ab}_j,c|_S\leq |p^{ab}_j,p^{ab}_\ell|_S+|p^{ab}_\ell,
p^{bc}_0|_S+|p^{bc}_0, c|_S\leq
\big((r+\epsilon)-(\mu+2D)\big)+D+\mu \leq r,$$ as desired.

In the first case, if $\ell> |a,c|_S$, then $p^{ac}_\ell$ lies
in (or is equal to) $c$ and hence
$$|p^{ab}_j,c|_S\leq |p^{ab}_j,p^{ab}_\ell|_S+|p^{ab}_\ell, p^{ac}_\ell|_S\leq \big((r+\epsilon)-(\mu+2D)\big)+D\leq
r.$$ If in the first case $\ell\leq |a,c|_S$, then
$$|p^{ab}_j,c|_S\leq |p^{ab}_j,p^{ab}_\ell|_S+|p^{ab}_\ell, p^{ac}_\ell|_S+|p^{ac}_\ell,
c|_S\leq \big(\ell-(\mu+2D)\big)+D+(r+\epsilon-\ell)\leq r.$$
\end{proof}

\subsection{Contractibility}
\label{sub:contr}

In this subsection we prove Theorem~\ref{thm:contractible}. To
do that, we use dismantlability.

\begin{definition}
We say that a vertex $a$ of a graph is \emph{dominated} by an
adjacent vertex $z\neq a$, if all the vertices adjacent to $a$
are also adjacent to $z$.

A finite graph is \emph{dismantlable} if its vertices can be
ordered into a sequence $a_1, \ldots, a_k$ so that for each
$i<k$ the vertex $a_i$ is dominated in the subgraph induced on
$\{a_i,\ldots,a_k\}$.
\end{definition}

Polat proved that the automorphism group of a dismantlable
graph fixes a clique \cite[Theorem A]{Po93} (for the proof see also \cite[Theorem~2.4]{HOP14}).
We will use the following strengthening of that result.

\begin{theorem}[{\cite[Theorem~6.5]{BaMi12},\cite[Theorem~1.2]{HOP14}}]
\label{thm:Damian} Let $\Gamma$ be a finite dismantlable graph.
Then for any subgroup $H\leq\mathrm{Aut}(\Gamma)$, the
fixed-point set $(\Gamma^\blacktriangle)^H$ is contractible.
\end{theorem}

Our key result is the dismantlability in the $n$-Rips graph.

\begin{lemma}
\label{lem:dismantlable} Let $U\subset V\cup W$ be a finite
subset that is $6D$-convex with respect to some $b\in U$. Then
for $n\geq 7D$, the subgraph of $\Gamma_n$ induced on $U$ is
dismantlable.
\end{lemma}

\begin{proof}
We order the vertices of $U$ according to $|\cdot,b|_S$,
starting from $a\in U$ with maximal $|a,b|_S$, and ending with
$b$.

We first claim that unless $U=\{b\}$, the set $U-\{a\}$ is
still $6D$-convex with respect to $b$. Indeed, let $u\in
U-\{a\}$ and let $(p_j)_{j=0}^\ell$ be a geodesic from $b$ to
$u$ in $\Gamma$. Let $j\leq \ell-6D$. Then $|p_j,b|_S\leq
\ell-6D<|a,b|_S$, so $p_j\neq a$ and hence $p_j\in U-\{a\}$
since $U$ was $6D$-convex. Similarly, if $w\in W$ and
$|w,p_j|_S\leq \epsilon$, then $|w,b|\leq
\epsilon+(\ell-6D)<|a,b|_S$, so $w\neq a$ and hence $w\in
U-\{a\}$. This justifies the claim.

It remains to prove that $a$ is dominated in the subgraph of
$\Gamma_n$ induced on $U$ by some vertex $z$. Let
$(p^{ab}_j)_{j=0}^\ell$ be a geodesic from $a$ to $b$ in
$\Gamma$. If $\ell\leq 7D$, then we can take $z=b$ and the
proof is finished. We will henceforth suppose $\ell>7D$. If
$p^{ab}_{6D}$ is an $(\epsilon, R)$-deep vertex of $p^{ab}$ in
the peripheral left coset $w$, then let $z=w$, otherwise let
$z=p^{ab}_{6D}$. Note that by definition of convexity, we have
$z\in U$. We will show that $z$ dominates~$a$.

Let $c\in U$ be adjacent to~$a$ in $\Gamma_n$, which means
$|a,c|_S\leq n$. We apply Theorem~\ref{thm:stronglythin} to
$a,b,c,i=6D$ and any geodesics $(p^{bc}_j),(p^{ac}_j)$.
Consider first the case where $|z,p^{ac}_{6D}|_S\leq D$. If
$6D\leq |a,c|_S$, then $$|c,z|_S\leq |c,p^{ac}_{6D}|_S
+|p^{ac}_{6D},z|_S\leq (n-6D)+D<n,$$ so $c$ is adjacent to $z$
in $\Gamma_n$, as desired. If $|a,c|_S<6D$, then $p^{ac}_{6D}$
lies in (or is equal to) $c$ and hence $|c,z|_S\leq D<n$ as
well.

Now consider the case where $|z, p^{bc}_{\ell-6D}|_S \leq D$.
Since $a$ was chosen to have maximal $|a,b|_S$, we have
$|a,b|_S\geq |c,b|_S$, and hence $|c,p^{bc}_{\ell-6D}|_S\leq
6D$. Consequently,
$$|c,z|_S\leq |c,p^{bc}_{\ell-6D}|+|p^{bc}_{\ell-6D},z|_S+\leq
6D+D\leq n.$$
\end{proof}

We are now ready to prove the contractibility of the
fixed-point sets.

\begin{proof}[Proof of Theorem~\ref{thm:contractible}]
Let $n\geq 7D$. Suppose that the fixed-point set
$\mathrm{Fix}=(\Gamma_n^\blacktriangle)^H$ is nonempty.

\smallskip

\noindent \textbf{Step 1.} \emph{The fixed-point set
$\mathrm{Fix'}=(\Gamma_{4D}^\blacktriangle)^H$ is nonempty.}

\smallskip

Let $U$ be the vertex set of a simplex in
$\Gamma_n^\blacktriangle$ containing a point of $\mathrm{Fix}$
in its interior. Note that $U$ is $H$-invariant. If $U$ is a
single vertex $u$, then $u\in\mathrm{Fix'}$ and we are done.
Otherwise, by Lemma~\ref{lem:quasicentre}, the quasi-centre of
$U$ spans a simplex in $\Gamma_{4D}^\blacktriangle$.
Consequently, its barycentre lies in $\mathrm{Fix}'$.

\smallskip

\noindent \textbf{Step 2.} \emph{If $\mathrm{Fix}'$ contains at
least $2$ points, then it contains a point outside $W$.}

Otherwise, choose $w\neq w'\in\mathrm{Fix}'$ with minimal
$|w,w'|_S$. If $|w,w'|_S\leq 4D$, then the barycentre of the
edge $ww'$ lies in $\mathrm{Fix}'$, which is a contradiction.
If $|w,w'|_S> 4D$, then $\rho(\{w,w'\})\leq
\big\lceil\frac{|w,w'|_S}{2}\big\rceil<|w,w'|_S$. Let $U'$ be
the quasi-centre of $\{w,w'\}$. By Lemma~\ref{lem:quasicentre},
we have that $U'$ spans a simplex in
$\Gamma_{4D}^\blacktriangle$, with barycentre in
$\mathrm{Fix}'$. If $U'$ is not a single vertex, this is a
contradiction. Otherwise, if $U'$ is a single vertex $w''\in
W$, then $|w,w''|_S\leq \rho(\{w,w'\})<|w,w'|_S$, which
contradicts our choice of $w,w'$.

\smallskip
\noindent \textbf{Step 3.} \emph{$\mathrm{Fix}$ is
contractible.}

\smallskip
By Step~1, the set $\mathrm{Fix}'$ is nonempty. If
$\mathrm{Fix}'$ consists of only one point, then so does
$\mathrm{Fix}$, and there is nothing to prove. Otherwise, let
$\Delta$ be the simplex in $\Gamma_{4D}^\blacktriangle$
containing in its interior the point of $\mathrm{Fix}'$
guaranteed by Step~2. Note that $\Delta$ is also a simplex in
$\Gamma_n^\blacktriangle$ with barycentre in $\mathrm{Fix}$.
Since $\Delta$ is not a vertex of $W$, by
Lemma~\ref{rem:finiteU} all its $r$-hulls $\Delta_r$ are
finite. By Lemma~\ref{lem:hull}, each $\Delta_r$ with $r\geq
4D$ is $6D$-convex. Thus by Lemma~\ref{lem:dismantlable}, the
$1$-skeleton of the span $\Delta^\blacktriangle_r$ of
$\Delta_r$ in $\Gamma_n^\blacktriangle$ is dismantlable. Hence
by Theorem~\ref{thm:Damian}, the set $\mathrm{Fix}\cap
\Delta^\blacktriangle_r$ is contractible. Note that $\Delta_r$
exhaust entire $V\cup W$. Consequently, entire $\mathrm{Fix}$
is contractible, as desired.
\end{proof}

\subsection{Edge-dismantlability}
\label{sub:edge-dis} Mineyev and Yaman introduced for a
relatively hyperbolic group a complex $X(\mathcal{G},\mu)$,
which they proved to be contractible \cite[Theorem 19]{MiYa07}.
However, analysing their proof, they exhaust the 1-skeleton of
$X(\mathcal{G},\mu)$ by finite graphs that are not dismantlable
but satisfy a slightly weaker relation, which we can call
\emph{edge-dismantlability}.

An edge $(a,b)$ of a graph is \emph{dominated} by a vertex $z$
adjacent to both $a$ and $b$, if all the other vertices
adjacent to both $a$ and $b$ are also adjacent to $z$. A finite
graph $\Gamma$ is \emph{edge-dismantlable} if there is a
sequence of subgraphs $\Gamma=\Gamma_1\supset \Gamma_2\supset
\cdots  \supset\Gamma_k$, where for each $i<k$ the graph
$\Gamma_{i+1}$ is obtained from $\Gamma_i$ by removing a
dominated edge or a dominated vertex with all its adjacent
edges, and where $\Gamma_k$ is a single vertex.

In dimension 2 the notion of edge-dismantlability coincides
with collapsibility, so by \cite{Se} the automorphism group of
$\Gamma$ fixes a clique, similarly as for dismantlable graphs.

\begin{question}
Does the automorphism group of an arbitrary edge-dismantlable graph
$\Gamma$ fix a clique? For arbitrary $H\leq\mathrm{Aut}(\Gamma)$, is the
fixed-point set $(\Gamma^\blacktriangle)^H$ contractible?
\end{question}

\section{Proof of the thin triangle theorem}
\label{sec:triangle}

\subsection{Preliminaries}

Given an edge-path $p=(p_i)_{i=0}^\ell$, we use the following
notation. The length $\ell$ of $p$ is denoted by $l(p)$, the
initial vertex $p_0$ of $p$ is denoted by $p_-$, and its
terminal vertex $p_\ell$ is denoted by $p_+$. For integers $0\leq
j\leq k\leq l(p)$, we denote by $p[j,k]$ the subpath
$(p_i)_{i=j}^k$, and by $p[k,j]$ we denote the inverse path.

The group $G$ is hyperbolic relative to $\mc P$ in the sense of Osin~\cite[Definition 1.6, Theorem 1.5]{Os06}. We first recall two results from~\cite{Os06, Os07}. Consider the alphabet
$\mathscr{P}=S\sqcup \bigsqcup_\lambda P_\lambda$. Every word
in this alphabet represents an element of $G$, and note that
distinct letters might represent the same element. Let $\bar
\Gamma$ denote the Cayley graph of $G$ with respect to
$\mathscr{P}$.

\begin{theorem}[{\cite[Theorem 3.26]{Os06}}]
\label{thm:OsinTriangle} There is a
constant $K>0$ with the following property. Consider a triangle
whose sides $p$, $q$, $r$ are geodesics in $\bar \Gamma$. For
any vertex $v$ of $p$, there exists a vertex $u$ of $q\cup r$
such that $|u,v|_S \leq K.$
\end{theorem}

\begin{definition}[{\cite[page 17]{Os06}}]
Let $q$ be an edge-path in $\bar \Gamma$. Subpaths of~$q$ with
at least one edge are called \emph{non-trivial}. A
\emph{$gP_\lambda$-component} of~$q$ is a maximal non-trivial
subpath $r$ such that the label of $r$ is a word in $P_\lambda
- \{1\}$ and a vertex of $r$ (and hence all its vertices)
belong to $gP_\lambda$.  We refer to $gP_\lambda$-components as
\emph{$\mc P$-components} if there is no need to specify
$gP_\lambda$. A $gP_\lambda$-component of $q$ is
\emph{isolated} if $q$ has no other $gP_\lambda$-components.
Note that $gP_\lambda$-components of geodesics in $\bar \Gamma$
are single edges and we call them  \emph{$gP_\lambda$-edges}.
\end{definition}

\begin{theorem}[{\cite[Proposition 3.2]{Os07}}]
\label{thm:n-gon} There is a constant $K>0$ satisfying the
following condition. Let $\Delta$ be an $n$-gon in
$\bar\Gamma$, which means that $\Delta$ is a closed path that
is a concatenation of $n$ edge-paths $\Delta=q^1  q^2  \dots
q^n$. Suppose that $I \subset \{1, \dots , n\}$ is such that
\begin{enumerate}
\item for $i \in I$ the side $q^i$  is an isolated $\mc P$-component of $\Delta$, and
\item for $i \not \in I$ the side $q^i$ is a geodesic in
    $\bar\Gamma$.
\end{enumerate}
Then $ \sum_{i \in I} |q^i_- , q^i_+|_S \leq Kn$.
 \end{theorem}

We now recall a result of Hruska~\cite{HK08} and another of
Dru\c{t}u--Sapir~\cite{DrSa08} on the relation between the geometry
of $\bar \Gamma$ and the Cayley graph $\Gamma$ of $G$ with
respect to $S$.

\begin{definition}
Let $p$ be a geodesic in $\Gamma$, and let $\epsilon, R$ be
positive integers. Let $w\in W$ be a peripheral left coset. An
\emph{$(\epsilon, R)$-segment in $w$} of $p$ is a maximal
subpath such that all its vertices are {$(\epsilon, R)$-deep}
in $w$. Note that an $(\epsilon, R)$-segment could consist of a
single vertex.
\end{definition}

\begin{definition}
Edge-paths $p$ and $q$ in $\bar \Gamma$ are \emph{$K$-similar}
if $|p_-,q_-|_S\leq K$ and $|p_+,q_+|_S\leq K$.
\end{definition}

\begin{proposition}[{see \cite[Proposition~8.13]{HK08}}]  \label{prop:transition}\label{cor:segments}
There are constants $\epsilon$, $R$ satisfying
Lemma~\ref{lem:uniqueness},
 and a constant $K$ such that the following holds. Let $p$ be a geodesic in $\Gamma$ and let $\bar p$ be a geodesic in $\bar \Gamma$ with the same endpoints as $p$.
 \begin{enumerate}[(i)]
 \item  The set of  vertices of $\bar p$ is at Hausdorff
     distance at most $K$ from the set of $(\epsilon,
     R)$-transition vertices of $p$, in the metric
     $|\cdot,\cdot|_S$.
 \item If $p[j,k]$ is an $(\epsilon, R)$-segment in $w$ of
     $p$, then there are vertices $\bar p_m$ and $\bar p_n$
     of $\bar p$ such that $|p_j, \bar p_m|_S\leq K$ and
     $|p_k, \bar p_n|_S \leq K$.
 \item For any subpath $\bar p[m,n]$ of $\bar p$ with
     $m\leq n$ there is a $K$-similar subpath $p[j,k]$
     of~$p$ with $j\leq k$.
   \end{enumerate}
\end{proposition}
\begin{proof}
The existence of $\epsilon$, $R$, and $K$ satisfying
Lemma~\ref{lem:uniqueness} and (i) is~\cite[Proposition
8.13]{HK08}, except that Hruska considers the set of transition
points instead of vertices. However, after increasing his $R$
by~$1$, we obtain the current statement, and moreover each pair
of distinct $(\epsilon, R)$-segments is separated by a
transition vertex. Consequently, by increasing $K$ by $1$, we
obtain (ii).

For the proof of (iii), increase $K$ so that it satisfies
Theorem~\ref{thm:OsinTriangle}. We will show that $3K$
satisfies statement (iii). By (i), there is a vertex $p_{k}$
such that $|\bar p_n, p_{k}|_S\leq K$. Let $q,\bar{q}$ be
geodesics in $\Gamma,\bar\Gamma$ from $\bar p_n$ to $p_{k}$.
Let $\bar r$ be a geodesic in $\bar \Gamma$ from $p_0$ to
$p_{k}$. Consider the geodesic triangle in $\bar \Gamma$ with
sides $\bar p[0, n]$, $\bar q$, and $\bar r$. By
Theorem~\ref{thm:OsinTriangle}, there is a vertex $v$ of $\bar
q \cup \bar r$ such that $|\bar p_m, v|_S\leq K$.

Suppose first that $v$ lies in $\bar r$. By (i) there is a
vertex $p_{j}$ of $p[0,k]$ such that $|p_{j},v|_S\leq K$. It
follows that $|p_{j}, \bar p_m|_S\leq 2K$. Now suppose that $v$
lies in $\bar q$. By (i) the vertex $v$ is at $S$-distance
$\leq K$ from a vertex of $q$. Since $l(q)\leq K$, it follows
that $|p_{k}, \bar p_m|_S \leq 3K$, and we can assign $j=k$.
\end{proof}

\begin{lemma}[{Quasiconvexity, \cite[Lemma 4.15]{DrSa08}}]
\label{lem:quasiconvexity}
There is $K>0$  such that the following holds. Let $w\in W$ be
a peripheral left coset, let $A$ be a positive integer, and let
$p$ be a geodesic in $\Gamma$ with $|p_-, w|_S\leq A$ and
$|p_+, w|_S\leq A$.  Then any vertex $p_i$ of $p$ satisfies
$|p_i, w|_S\leq KA$.
\end{lemma}

\begin{assumption} \label{constants}
From here on, the constants $(\epsilon, R, K)$ are assumed to
satisfy the statement of Proposition~\ref{prop:transition}. By
increasing $K$, we also assume that $K$ satisfies the
conclusions of Theorems~\ref{thm:OsinTriangle}
and~\ref{thm:n-gon}, the quasiconvexity
Lemma~\ref{lem:quasiconvexity}, and $K\geq \max\{\epsilon,
R\}$.
\end{assumption}

We conclude with the following application of
Theorem~\ref{thm:n-gon}.

\begin{lemma} \label{lem:deep2}
Let $p$ be a geodesic in $\Gamma$, and let $\bar p$ be a
geodesic in $\bar \Gamma$ with the same endpoints as $p$. Let
$p[j,k]$ be an $(\epsilon, R)$-segment of a peripheral left
coset $w\in W$. If $k-j>8K^2$, then $\bar p$ contains a
$w$-edge which is $5K^2$-similar to $p[j,k]$.
\end{lemma}
\begin{proof}
There are vertices $r_-$ and $r_+$ of $w$ such that
$|p_j,r_-|_S\leq \epsilon\leq K$ and $|p_k, r_+|_S \leq
\epsilon\leq K.$ Let $r$ be a $w$-edge from $r_-$ to~$r_+$. By
Proposition~\ref{prop:transition}(ii), there are vertices
$\bar{p}_m,\bar{p}_n$ at $S$-distance at most $K$ from
$p_j,p_k$, respectively. Let $[r_-, \bar p_m]$ and $[\bar p_n,
r_-]$ be geodesics in $\Gamma$ between the corresponding
vertices; note that the labels of these paths are words in the
alphabet $S$, and they both have length at most~$2K$.

Suppose for contradiction that $\bar p$ does not contain a
$w$-edge. Consider the closed path $[r_-, \bar p_m]  \bar
p[m,n]  [\bar p_n, r_+] r$, viewed as a polygon $\Delta$
obtained by subdividing $[r_-, \bar p_m]$ and $[\bar p_n, r_+]$
into edges. Since $\bar p[m,n]$ is a geodesic in $\bar \Gamma$,
the number of sides of $\Delta$ is at most $2+|r_-,\bar
p_m|_S+|r_+,\bar p_n|_S \leq 6K$. We have that $r$ is an
isolated $w$-component of $\Delta$. Then
Theorem~\ref{thm:n-gon} implies that $|r_-, r_+|_S \leq 6K^2$.
It follows that $|p_j, p_k|_S \leq 8K^2$. This is a
contradiction, hence $\bar p[m,n]$ contains a $w$-edge $t$.

Now we prove that $|t_-, p_j|_S\leq 5K^2$. Let $[\bar p_m,
t_-]$ be the subpath of $\bar p[m,n]$ from $\bar p_m$ to $t_-$,
and note that this is a geodesic in $\bar \Gamma$ without
$w$-components. Let $[t_-,r_-]$ be a $w$-edge from $t_-$ to
$r_-$. Consider the polygon $[\bar p_m, t_-] [t_-,r_-]
[r_-,\bar p_m]$, where the path $[r_-,\bar p_m]$ is subdivided
into at most $2K$ edges. Observe that $[t_-,r_-]$ is an
isolated $w$-component. Theorem~\ref{thm:n-gon} implies that
$|t_-, p_j|_S \leq |t_-, r_-|_S+|r_-, p_j|_S \leq K(2K+2)+K\leq
5K^2$.

Analogously one proves that $|t_+, p_k|_S\leq 5K^2$.
\end{proof}

\subsection{Proof}

We are now ready to start the proof of
Theorem~\ref{thm:stronglythin}. Let $D=53K^2$.

Let $a,b,c\in V\cup W$ with $a\neq b$, and let $p^{ab}, p^{bc},
p^{ac}$ be geodesics in $\Gamma$ from $a$ to $b$, from $b$ to
$c$, and from $a$ to $c$, respectively. Let $\ell=|a,b|_S$ and
let $0\leq i \leq \ell $. If $p^{ab}_i$ is an $(\epsilon,
R)$-deep vertex of $p^{ab}$ in the peripheral left coset~$w$
then let $z=w$, otherwise let $z=p^{ab}_i$.

We define the following paths illustrated in
Figure~\eqref{fig:hexagon}.
\begin{figure}
\begin{picture}(0,0)%
\includegraphics{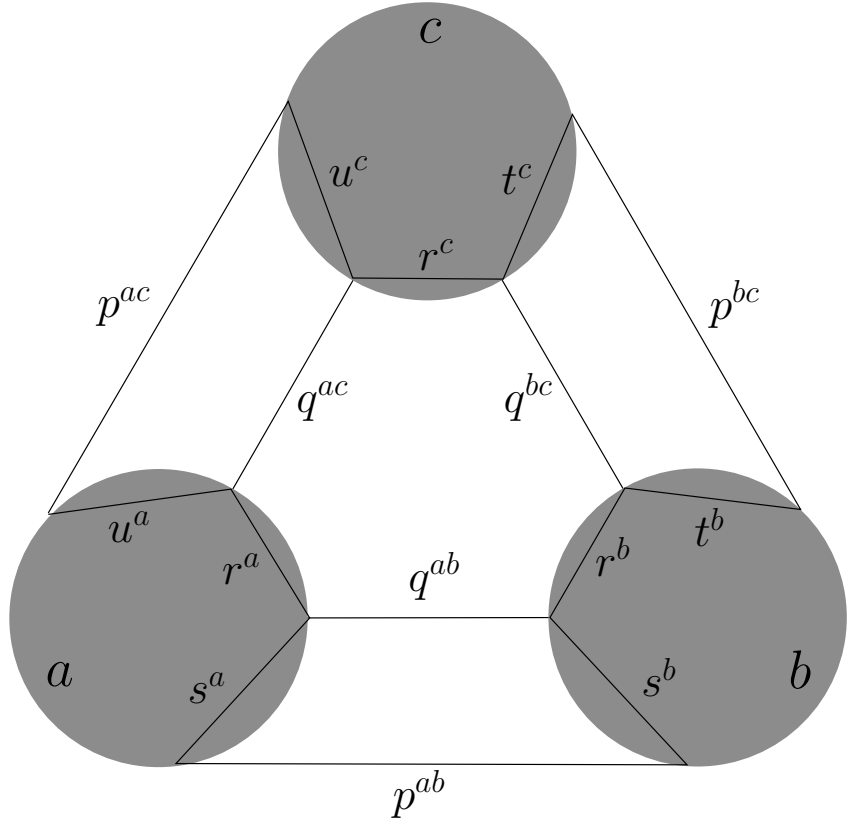}%
\end{picture}%
%
%
\setlength{\unitlength}{3947sp}%
\begingroup\makeatletter\ifx\SetFigFont\undefined%
\gdef\SetFigFont#1#2#3#4#5{%
  \reset@font\fontsize{#1}{#2pt}%
  \fontfamily{#3}\fontseries{#4}\fontshape{#5}%
  \selectfont}%
\fi\endgroup%
\begin{picture}(5261,5219)(905,-4708)
\put(4011,-2106){\makebox(0,0)[lb]{\smash{{\SetFigFont{17}{20.4}{\rmdefault}{\mddefault}{\updefault}{\color[rgb]{0,0,0}$q^{bc}$}%
}}}}
\put(3314,-4600){\makebox(0,0)[lb]{\smash{{\SetFigFont{17}{20.4}{\rmdefault}{\mddefault}{\updefault}{\color[rgb]{0,0,0}$p^{ab}$}%
}}}}
\put(1460,-1509){\makebox(0,0)[lb]{\smash{{\SetFigFont{17}{20.4}{\rmdefault}{\mddefault}{\updefault}{\color[rgb]{0,0,0}$p^{ac}$}%
}}}}
\put(5305,-1518){\makebox(0,0)[lb]{\smash{{\SetFigFont{17}{20.4}{\rmdefault}{\mddefault}{\updefault}{\color[rgb]{0,0,0}$p^{bc}$}%
}}}}
\put(3412,-3228){\makebox(0,0)[lb]{\smash{{\SetFigFont{17}{20.4}{\rmdefault}{\mddefault}{\updefault}{\color[rgb]{0,0,0}$q^{ab}$}%
}}}}
\put(2709,-2103){\makebox(0,0)[lb]{\smash{{\SetFigFont{17}{20.4}{\rmdefault}{\mddefault}{\updefault}{\color[rgb]{0,0,0}$q^{ac}$}%
}}}}
\put(2027,-3934){\makebox(0,0)[lb]{\smash{{\SetFigFont{17}{20.4}{\rmdefault}{\mddefault}{\updefault}{\color[rgb]{0,0,0}$s^a$}%
}}}}
\put(4877,-3892){\makebox(0,0)[lb]{\smash{{\SetFigFont{17}{20.4}{\rmdefault}{\mddefault}{\updefault}{\color[rgb]{0,0,0}$s^b$}%
}}}}
\put(5209,-2951){\makebox(0,0)[lb]{\smash{{\SetFigFont{17}{20.4}{\rmdefault}{\mddefault}{\updefault}{\color[rgb]{0,0,0}$t^b$}%
}}}}
\put(3998,-736){\makebox(0,0)[lb]{\smash{{\SetFigFont{17}{20.4}{\rmdefault}{\mddefault}{\updefault}{\color[rgb]{0,0,0}$t^c$}%
}}}}
\put(2910,-688){\makebox(0,0)[lb]{\smash{{\SetFigFont{17}{20.4}{\rmdefault}{\mddefault}{\updefault}{\color[rgb]{0,0,0}$u^c$}%
}}}}
\put(1525,-2921){\makebox(0,0)[lb]{\smash{{\SetFigFont{17}{20.4}{\rmdefault}{\mddefault}{\updefault}{\color[rgb]{0,0,0}$u^a$}%
}}}}
\put(2242,-3185){\makebox(0,0)[lb]{\smash{{\SetFigFont{17}{20.4}{\rmdefault}{\mddefault}{\updefault}{\color[rgb]{0,0,0}$r^a$}%
}}}}
\put(4586,-3134){\makebox(0,0)[lb]{\smash{{\SetFigFont{17}{20.4}{\rmdefault}{\mddefault}{\updefault}{\color[rgb]{0,0,0}$r^b$}%
}}}}
\put(3473,-1212){\makebox(0,0)[lb]{\smash{{\SetFigFont{17}{20.4}{\rmdefault}{\mddefault}{\updefault}{\color[rgb]{0,0,0}$r^c$}%
}}}}
\put(3474,216){\makebox(0,0)[lb]{\smash{{\SetFigFont{20}{24.0}{\rmdefault}{\mddefault}{\updefault}{\color[rgb]{0,0,0}$c$}%
}}}}
\put(1133,-3828){\makebox(0,0)[lb]{\smash{{\SetFigFont{20}{24.0}{\rmdefault}{\mddefault}{\updefault}{\color[rgb]{0,0,0}$a$}%
}}}}
\put(5799,-3843){\makebox(0,0)[lb]{\smash{{\SetFigFont{20}{24.0}{\rmdefault}{\mddefault}{\updefault}{\color[rgb]{0,0,0}$b$}%
}}}}
\end{picture}%
\caption{Paths in the proof of the thin triangle theorem}\label{fig:hexagon}
\end{figure}
Let $\bar p^{ab}, \bar p^{bc}, \bar p^{ac}$ be geodesics in
$\bar \Gamma$ with the same endpoints as $p^{ab}$, $p^{bc}$,
$p^{ac}$, respectively. If $a\in W$ and $\bar{p}^{ab}$ starts
with an $a$-edge, then we call this edge $s^a$; otherwise let
$s^a$ be the trivial path. We define $s^b$ analogously. Then
$\bar p^{ab}$ is a concatenation
\[ \bar p^{ab} = s^a q^{ab} s^b. \]
Similarly, the paths $\bar p^{ac},\bar p^{bc}$ can be expressed
as concatenations
\begin{equation} \nonumber \bar p^{ac} = u^a q^{ac} u^c, \quad \quad  \bar p^{bc} = t^b q^{bc} t^c.\end{equation}
Let $r^a$ be a path in $\bar{\Gamma}$ from $u^a_+$ to $s^a_+$
that is a single $a$-edge if $u^a_+\neq s^a_+$, or the trivial
path otherwise. We define $r^b, r^c$ analogously. Let $\Pi$ be
the geodesic hexagon in $\bar \Gamma$ given by
\[ \Pi = r^a q^{ab} r^b q^{bc} r^c q^{ca} .\]

\begin{step}
\label{step:first}
Paths $p^{ab}$ and $q^{ab}$ are $2K$-similar. The same is true
for the pair $p^{ac}$ and $q^{ac}$, and the pair $p^{bc}$ and
$q^{bc}$. \label{eq:01}
\end{step}

\begin{proof}
By Proposition~\ref{prop:transition}(i), there is a vertex
$p^{ab}_j$ such that $|s^a_+, p^{ab}_j|_S \leq K$.  Since
$p^{ab}$ is a geodesic from $a$ to $b$ in $\Gamma$, it follows
that $j\leq K$, and consequently $|s^a_-, s^a_+|_S \leq 2K$.
The remaining assertions are proved analogously.
\end{proof}

In view of Proposition~\ref{prop:transition}(i), there is a
vertex of $\bar p^{ab}$ at $S$-distance $\leq K$ from
$p^{ab}_i$. While this vertex might be $s_-^a$ or $s_-^b$,
Step~\ref{step:first} guarantees that there is a vertex
$q^{ab}_h$ at $S$-distance $\leq 3K$ from $p^{ab}_i$.

The following step should be considered as the bigon case of
Theorem~\ref{thm:stronglythin}.

\begin{step}
\label{Step:combined} If $q^{ac}$ contains a
vertex in $b$, in the case where $b\in W$, or equal to $b$, in the case where $b\in V$, then $|z , p^{ac}_i |_S< D$.

Similarly, if  $q^{bc}$ contains a vertex in $a$, in the case where $a\in W$, or equal to $a$, in the case where $a\in V$, then $|z ,
p^{bc}_{\ell-i} |_S< D$.
\end{step}

Note that here we keep the convention that for $i>l(p^{ac})$
the vertex $p^{ac}_i$ denotes $p^{ac}_+$, and similarly if
$\ell-i>l(p^{bc})$, then the vertex $p^{bc}_{\ell-i}$ denotes
$p^{bc}_+$.

\begin{proof}
By symmetry, it suffices to prove the first assertion. We focus on the case $b\in W$,   the case $b\in V$ follows by considering $x$ below as the trivial path.

Let $q_n^{ac}$ be the first vertex of $q^{ac}$ in $b$. Let $x$
be a $b$-edge joining $q^{ab}_+$ to $q_n^{ac}$. The edges $r^a$
and $x$ are isolated $\mc P$-components of the $4$-gon
$r^aq^{ab}xq^{ac}[n,0]$. By Theorem~\ref{thm:n-gon}, we have
$|r^a_-,r^a_+|_S,|x_-,x_+|_S\leq 4K$. Consequently by
Step~\ref{step:first} we have $|p_-^{ab},p_-^{ac}|_S\leq 8K$.

First consider the case, where $p^{ab}_i$ is an $(\epsilon,
R)$-transition vertex of $p^{ab}$. Let $q^{ab}_h$ be the vertex
defined after Step~\ref{step:first}. Subdividing the $4$-gon
$r^aq^{ab}xq^{ac}[n,0]$ into two geodesic triangles in $\bar
\Gamma$, and applying twice Theorem~\ref{thm:OsinTriangle},
gives $h^*$ such that $|q^{ab}_h, q^{ac}_{h^*}|_S \leq 6K$. By
Proposition~\ref{prop:transition}(i), there is $i^*$ such that
$|q^{ac}_{h^*}, p^{ac}_{i^*}|_S \leq K$. It follows that
$|p^{ab}_i, p^{ac}_{i^*}|_S \leq 3K+6K+K=10K$ and hence
$|i-i^*| \leq 8K+10K=18K$. Therefore $|p^{ab}_i, p^{ac}_i|_S
\leq 10K+18K$.

Now consider the case, where $p^{ab}_i$ is an $(\epsilon,
R)$-deep vertex of $p^{ab}$ in the peripheral left coset $w$.
Then $p^{ab}_i$ lies in an $(\epsilon, R)$-segment
$p^{ab}[j,k]$ of $p^{ab}$ in $w$. Thus $\max\{|p^{ab}_j, w|_S,
|p^{ab}_k,w|_S\} \leq \epsilon \leq K.$ By
Proposition~\ref{cor:segments}(ii) and Step~\ref{step:first},
there is a vertex $q^{ab}_m$ such that $|p^{ab}_j,
q^{ab}_m|_S\leq 3K$. As in the previous case, we obtain $j^*$
such that $|p^{ab}_j, p^{ac}_{j^*}|_S \leq 10K$. Analogously,
there is $k^*$ such that $|p^{ab}_k, p^{ac}_{k^*}|_S \leq 10K$.
In particular, we have $\max\left\{|p^{ac}_{j^*},w|_S,
|p^{ac}_{k^*},w|_S \right\}\leq 11K$. By quasiconvexity of $w$,
Lemma~\ref{lem:quasiconvexity}, every vertex of
$p^{ac}[j^*,k^*]$ is at $S$-distance $\leq 11K^2$ from $w$.
Moreover, we have $|j-j^*|\leq 8K+10K$, and analogously,
$|k-k^*|\leq 18K$. Hence $i$ is at distance $\leq 18K$ from the
interval $[j^*,k^*]$. (Note that we might have $j^*>k^*$, but
that does not change the reasoning.) It follows that
$|w,p^{ac}_i|_S \leq 11K^2+18K.$
\end{proof}

By Step~\ref{Step:combined}, we can assume that $a\neq c$ and
$b\neq c$. Moreover, we can assume that there is no
$b$-component in $q^{ac}$, nor an $a$-component in $q^{bc}$.
Consequently,  we can apply Theorem~\ref{thm:n-gon} to $\Pi$,
viewing $r^a$ and $r^b$ as isolated components and the
remaining four sides as geodesic sides. It follows that
\begin{equation} \label{eq:added}
\max\{|r^a_-, r^a_+|_S ,  |r^b_-, r^b_+|_S \} \leq 6K.
\end{equation}
Together with Step~\ref{eq:01}, this implies
\begin{equation} \label{eq:02} \max\{|p^{ab}_-, p^{ac}_-|_S, |p^{ab}_+, p^{bc}_-|_S\} \leq 10K.\end{equation}

\begin{step}\label{lem:triangle2}
If $p^{ab}_i$ is an $(\epsilon, R)$-transition vertex of
$p^{ab}$ or an endpoint of an $(\epsilon, R)$-segment of
$p^{ab}$, then $|p^{ab}_i , p^{ac}_i |_S\leq 34K$ or
 $|p^{ab}_i , p^{bc}_{\ell-i} |_S \leq 34K$.
\end{step}
\begin{proof}

By Step~\ref{step:first}, the vertex $p^{ab}_i$ is at
$S$-distance $\leq 3K$ from some $q_h^{ab}$. We split the
hexagon $\Pi$ into four geodesic triangles in $\bar \Gamma$
using diagonals. Repeated application of
Theorem~\ref{thm:OsinTriangle} and inequality~\eqref{eq:added}
yields a vertex of $q^{ac}\cup q^{bc}$ at $S$-distance $\leq
8K$ from $q_h^{ab}$. Without loss of generality, suppose that
this vertex is $q^{ac}_{h^*}$. By
Proposition~\ref{prop:transition}(i), the vertex $q^{ac}_{h^*}$
is at $S$-distance $\leq K$ from some $p^{ac}_{i^*}$.
Consequently, $|p^{ab}_i,p^{ac}_{i^*}|_S\leq 12K$. By
inequality~\eqref{eq:02}, it follows that $|i-i^*|\leq 22K$.
Therefore $|p^{ab}_i, p^{ac}_i|_S\leq 34K$.
\end{proof}

To prove Theorem~\ref{thm:stronglythin}, it remains to consider
the case where $p^{ab}_i$ is an $(\epsilon, R)$-deep vertex of
$p$ in a peripheral left coset $w$. Let $p^{ab}[j,k]$ be the
$(\epsilon, R)$-segment of $p^{ab}$ in $w$ containing
$p^{ab}_i$.

Note that if $w=c$, then $|a,c|_S\leq j+\epsilon$. Consequently
$p_i^{ac}$ is at $S$-distance $\leq \epsilon$ from $p_+^{ac}\in w$,
and the theorem follows. Henceforth, we will assume $w\neq c$.

Also note that if $k-j \leq 18K^2$, then assuming without loss
of generality that Step~\ref{lem:triangle2} yields
$|p^{ab}_j,p^{ac}_j|_S\leq 34K$, we have:
$$|w,p^{ac}_i|_S\leq
|w,p^{ab}_j|_S+|p^{ab}_j,p^{ac}_j|_S+|p^{ac}_j,p^{ac}_i|_S\leq
\epsilon+34K+18K^2\leq D,$$ and the theorem is proved.
Henceforth, we assume $k-j> 18K^2$.

\begin{step}\label{lem:lastcases}
The path $q^{ab}$ has a $w$-component $q^{ab}[m,m+1]$ which is
$5K^2$-similar to $p^{ab}[j,k]$. Moreover, $q^{bc}$ or $q^{ac}$
has a $w$-component.
\end{step}
\begin{proof}
The first assertion follows from Lemma~\ref{lem:deep2}. In particular, $w$ is distinct from $a$ and $b$. For the
second assertion, suppose for contradiction that $q^{bc}$ and $q^{ac}$
have no $w$-components.
Consequently, since $w\neq c$, the $w$-edge $q^{ab}[m,m+1]$ is an isolated $w$-component of $\Pi$. We  apply
Theorem~\ref{thm:n-gon} with $\Pi$ interpreted as an 8-gon with a side $q^{ab}[m,m+1]$ as the only isolated component.    This contradicts $|q^{ab}_m,
q^{ab}_{m+1}|_S> 18K^2-2\cdot 5K^2\geq 8K$.
\end{proof}

\begin{step}\label{lem:another-case}
Suppose that $q^{ac}$ has a $w$-component $q^{ac}[n,n+1]$  and
$q^{bc}$ does not have a $w$-component. Then $|w , p^{ac}_i
|_S< D$. Similarly, if $q^{bc}$ has a $w$-component and
$q^{ac}$ does not have a $w$-component, then $|w ,
p^{bc}_{\ell-i} |_S< D$.
\end{step}

\begin{proof}
By symmetry, it suffices to prove the first assertion. Let $x$
be a $w$-edge in $\bar{\Gamma}$ from $q^{ab}_m$ to $q^{ac}_n$.
Consider the geodesic $4$-gon $r^aq^{ab}[0,m]xq^{ac}[n,0]$ in
$\bar{\Gamma}$. Observe that $x$ is an isolated $w$-component
and hence Theorem~\ref{thm:n-gon} implies that
$|q^{ab}_m,q^{ac}_n|_S \leq 4K$. Analogously, by considering a
geodesic $6$-gon, we obtain
$|q^{ab}_{m+1},q^{ac}_{n+1}|_S \leq 6K$.

Proposition~\ref{prop:transition}(iii) implies that
$q^{ac}[n,n+1]$ is $K$-similar to a subpath $p^{ac}[j^*,k^*]$
of $p^{ac}$. By Step~\ref{lem:lastcases}, the paths
$p^{ab}[j,k]$ and $p^{ac}[j^*,k^*]$ are $(5K^2+6K+K)$-similar,
hence $12K^2$-similar. By inequality~\eqref{eq:02}, it follows
that $|j-j^*|\leq 10K+12K^2\leq 22K^2$ and similarly
$|k-k^*|\leq 22K^2$. Hence $i\in [j^*-22K^2,k^*+22K^2]$. By
Lemma~\ref{lem:quasiconvexity}, we have $|p_i^{ac},w|_S\leq
K\cdot K+22K^2$.
\end{proof}

It remains to consider the case where both $q^{ac}$ and
$q^{bc}$ have $w$-components, which we denote by
$q^{ac}[n,n+1]$ and $q^{bc}[\tilde n,\tilde n+1]$.

The argument in the proof of Step~\ref{lem:another-case} shows
that
\begin{equation} \label{eq:06}\max\{ |q^{ab}_m,q^{ac}_n|_S,   |q^{ac}_{n+1}, q^{bc}_{\tilde n+1}|_S,
|q^{ab}_{m+1},q^{bc}_{\tilde n}|_S   \} \leq 4K.\end{equation}

By Proposition~\ref{cor:segments}(iii) there are integers
$0\leq\alpha\leq\beta\leq l(p^{ac})$ and $0\leq \gamma\leq \delta \leq l(p^{bc})$ such that $q^{ac}[n,n+1]$ and
$p^{ac}[\alpha, \beta]$ are $K$-similar, and $q^{bc}[\tilde n,\tilde n+1]$
and $p^{bc}[\gamma, \delta]$ are $K$-similar.

\begin{step}\label{lem:final}
We have
\[ \alpha -20K^2 \leq i \leq \beta +33K^2
\quad \text{ or  } \quad \gamma-20K^2 \leq \ell-i \leq \delta+33K^2.\]
\end{step}
\begin{proof}
By~inequality~\eqref{eq:06} and Step~\ref{lem:lastcases}, we
have the following estimates:
\begin{align*}
 |p^{ac}_\beta ,
p^{bc}_\delta|_S &\leq |p^{ac}_\beta, q^{ac}_{n+1}|_S +
|q^{ac}_{n+1}, q^{bc}_{\tilde n+1}|_S  + |q^{bc}_{\tilde n+1},
p^{bc}_\delta|_S \leq 6K,\\
|p^{ab}_j, p^{ac}_\alpha|_S &\leq |p^{ab}_j, q^{ab}_m|_S +
|q^{ab}_{m}, q^{ac}_{n}|_S  + |q^{ac}_{n}, p^{ac}_\alpha|_S
\leq  5K^2+4K+K\leq 10K^2,\\
|p^{ab}_k, p^{bc}_\gamma|_S &\leq |p^{ab}_k, q^{ab}_{m+1}|_S +
|q^{ab}_{m+1}, q^{bc}_{\tilde n}|_S  + |q^{bc}_{\tilde n},
p^{bc}_\gamma|_S \leq  10K^2.
\end{align*}
These estimates and the triangle inequality imply
\begin{equation}\label{eq:07} k-j \leq (\beta-\alpha) + (\delta-\gamma) +26K^2.\end{equation}

From inequality~\eqref{eq:02} it follows that $|j-\alpha|\leq
10K+10K^2\leq 20K^2$, and analogously, $|(\ell-k)-\gamma|\leq
20K^2$. In particular, $\alpha - 20K^2 \leq j\leq i$ and
$\gamma - 20K^2 \leq \ell -k\leq \ell -i,$ as desired. To
conclude the proof we argue by contradiction. Suppose that $i >
\beta +33K^2$ and $\ell-i > \delta+33K^2$. It follows that
\[i-j\geq i-\alpha- 20K^2 > \beta -\alpha+13K^2\]
and
\[k-i =(\ell-i)-(\ell-k) \geq (\ell-i)-\gamma-20K^2> \delta -\gamma+13K^2.\]
Adding these two inequalities yields $k-j> (\beta - \alpha) +
(\delta -\gamma) +26K^2$, which contradicts
inequality~\eqref{eq:07}.
\end{proof}

We now conclude the proof of Theorem~\ref{thm:stronglythin}.
Since the endpoints of $q^{ac}[n,n+1]$ are in $w$, it follows
that the endpoints of $p^{ac}[\alpha, \beta]$ are at
$S$-distance $\leq K$ from $w$. By quasiconvexity of $w$,
Lemma~\ref{lem:quasiconvexity}, all the vertices of
$p^{ac}[\alpha, \beta]$ are at $S$-distance $\leq K^2$
from~$w$. Analogously, all the vertices of $p^{bc}[\gamma,
\delta]$ are at $S$-distance $\leq K^2$ from~$w$. Then
Step~\ref{lem:final} yields $|w , p^{ac}_i |_S\leq K^2+33K^2<
D$ or $|w , p^{bc}_{\ell-i} |_S \leq 34K^2< D$.


\begin{thebibliography}{10}

\bibitem{ABC91}
J.~M. Alonso, T.~Brady, D.~Cooper, V.~Ferlini, M.~Lustig, M.~Mihalik,
  M.~Shapiro, and H.~Short.
\newblock Notes on word hyperbolic groups.
\newblock In {\em Group theory from a geometrical viewpoint ({T}rieste, 1990)},
  pages 3--63. World Sci. Publ., River Edge, NJ, 1991.
\newblock Edited by Short.

\bibitem{BaMi12}
Jonathan~Ariel Barmak and Elias~Gabriel Minian.
\newblock Strong homotopy types, nerves and collapses.
\newblock {\em Discrete Comput. Geom.}, 47(2):301--328, 2012.

\bibitem{Bo12}
B.~H. Bowditch.
\newblock Relatively hyperbolic groups.
\newblock {\em Internat. J. Algebra Comput.}, 22(3):1250016, 66, 2012.

\bibitem{Br67}
Glen~E. Bredon.
\newblock {\em Equivariant cohomology theories}.
\newblock Lecture Notes in Mathematics, No. 34. Springer-Verlag, Berlin-New
  York, 1967.

\bibitem{BrHa99}
Martin~R. Bridson and Andr{\'e} Haefliger.
\newblock {\em Metric spaces of non-positive curvature}, volume 319 of {\em
  Grundlehren der Mathematischen Wissenschaften [Fundamental Principles of
  Mathematical Sciences]}.
\newblock Springer-Verlag, Berlin, 1999.

\bibitem{ChOs15}
Victor Chepoi and Damian Osajda.
\newblock Dismantlability of weakly systolic complexes and applications.
\newblock {\em Trans. Amer. Math. Soc.}, 367(2):1247--1272, 2015.

\bibitem{Da03}
Fran{\c{c}}ois Dahmani.
\newblock {\em Les groupes relativement hyperboliques et leurs bords}.
\newblock PhD thesis, 2003.


\bibitem{Da03-2}
Fran{\c{c}}ois Dahmani.
\newblock Classifying spaces and boundaries for relatively hyperbolic groups.
\newblock {\em Proc. London Math. Soc. (3)}, 86(3):666--684, 2003.


\bibitem{DrSa08}
Cornelia Dru{\c{t}}u and Mark~V. Sapir.
\newblock Groups acting on tree-graded spaces and splittings of relatively
  hyperbolic groups.
\newblock {\em Adv. Math.}, 217(3):1313--1367, 2008.

\bibitem{Fl98}
Jeffrey~L. Fletcher.
\newblock {\em Homological Group Invariants}.
\newblock PhD thesis, 1998.

\bibitem{Ge96}
S.~M. Gersten.
\newblock Subgroups of word hyperbolic groups in dimension {$2$}.
\newblock {\em J. London Math. Soc. (2)}, 54(2):261--283, 1996.

\bibitem{HaMa15}
Richard~Gaelan Hanlon and Eduardo Mart\'inez-Pedroza.
\newblock A subgroup theorem for homological filling functions.
\newblock {\em Groups Geom. Dyn.}
\newblock (To appear).

\bibitem{HOP14}
Sebastian Hensel, Damian Osajda, and Piotr Przytycki.
\newblock Realisation and dismantlability.
\newblock {\em Geom. Topol.}, 18(4):2079--2126, 2014.

\bibitem{HK08}
G.~Christopher Hruska.
\newblock Relative hyperbolicity and relative quasiconvexity for countable
  groups.
\newblock {\em Algebr. Geom. Topol.}, 10(3):1807--1856, 2010.

\bibitem{La13}
Urs Lang.
\newblock Injective hulls of certain discrete metric spaces and groups.
\newblock {\em J. Topol. Anal.}, 5(3):297--331, 2013.

\bibitem{Lu89}
Wolfgang L{\"u}ck.
\newblock {\em Transformation groups and algebraic {$K$}-theory}, volume 1408
  of {\em Lecture Notes in Mathematics}.
\newblock Springer-Verlag, Berlin, 1989.
\newblock Mathematica Gottingensis.

\bibitem{LuMe00}
Wolfgang L{\"u}ck and David Meintrup.
\newblock On the universal space for group actions with compact isotropy.
\newblock In {\em Geometry and topology: {A}arhus (1998)}, volume 258 of {\em
  Contemp. Math.}, pages 293--305. Amer. Math. Soc., Providence, RI, 2000.

\bibitem{MP15-2}
Eduardo Mart\'inez-Pedroza.
\newblock Subgroups of relatively hyperbolic groups of bredon cohomological
  dimension $2$.
\newblock arXiv1508.04865., 2015.

\bibitem{MP15}
Eduardo Mart{\'{\i}}nez-Pedroza.
\newblock A {N}ote on {F}ine {G}raphs and {H}omological {I}soperimetric
  {I}nequalities.
\newblock {\em Canad. Math. Bull.}, 59(1):170--181, 2016.

\bibitem{MW11}
Eduardo Mart{\'{\i}}nez-Pedroza and Daniel~T. Wise.
\newblock Relative quasiconvexity using fine hyperbolic graphs.
\newblock {\em Algebr. Geom. Topol.}, 11(1):477--501, 2011.

\bibitem{MS02}
David Meintrup and Thomas Schick.
\newblock A model for the universal space for proper actions of a hyperbolic
  group.
\newblock {\em New York J. Math.}, 8:1--7 (electronic), 2002.

\bibitem{MiYa07}
Igor Mineyev and Asli Yaman.
\newblock Relative hyperbolicity and bounded cohomology.
\newblock Available at http://www.math.uiuc.edu/~mineyev/math/art/rel-hyp.pdf,
  2007.

\bibitem{Os06}
Denis~V. Osin.
\newblock Relatively hyperbolic groups: intrinsic geometry, algebraic
  properties, and algorithmic problems.
\newblock {\em Mem. Amer. Math. Soc.}, 179(843):vi+100, 2006.

\bibitem{Os07}
Denis~V. Osin.
\newblock Peripheral fillings of relatively hyperbolic groups.
\newblock {\em Invent. Math.}, 167(2):295--326, 2007.

\bibitem{Po93}
N.~Polat.
\newblock Finite invariant simplices in infinite graphs.
\newblock {\em Period. Math. Hungar.}, 27(2):125--136, 1993.

\bibitem{Se}
Yoav Segev.
\newblock Some remarks on finite {$1$}-acyclic and collapsible complexes.
\newblock {\em J. Combin. Theory Ser. A}, 65(1):137--150, 1994.

\bibitem{tD87}
Tammo tom Dieck.
\newblock {\em Transformation groups}, volume~8 of {\em de Gruyter Studies in
  Mathematics}.
\newblock Walter de Gruyter \& Co., Berlin, 1987.

\end{thebibliography}
\end{document}